\documentclass[11pt,a4paper,oneside,english]{amsart}
\usepackage[T1]{fontenc}
\usepackage[latin9]{inputenc}

\usepackage{amsthm}
\usepackage{graphicx}
\usepackage{amssymb}
\usepackage{color}

\usepackage{babel}

\usepackage{fancyhdr}
\newtheorem{theorem}{Theorem}[section]
\newtheorem{lemma}{Lemma}[section]
\newtheorem{remark}{Remark}[section]

\newtheorem{definition}{Definition}[section]

\begin{document}


\title{Stability of a stochastically perturbed model of intracellular single-stranded RNA virus replication}\footnote{Preprint of paper published in \emph{Journal of Biological Systems} \textbf{27}(1), 69--82 (2019).\\
DOI: 10.1142/S0218339019500049 }

\author[L. Shaikhet, S.F. Elena and A. Korobeinikov]
{Leonid Shaikhet, Santiago F. Elena and Andrei Korobeinikov$^{*}$}

\maketitle


\medskip
{\footnotesize
\centerline{Department of Mathematics, Ariel University, Ariel 40700, Israel}
\centerline{leonid.shaikhet@usa.net}}

\medskip
{\footnotesize
\centerline{Instituto de Biolog\'{i}a Integrativa de Sistemas (I2SysBio), }
\centerline{Consejo Superior de Investigaciones Cient\'{i}ficas-Universitat de Val\`{e}ncia,}
\centerline{Catedr\'{a}tico Agust\'{i}n Escardino 9, 46980 Paterna, Val\`{e}ncia, Spain}
\centerline{and}
\centerline{Santa Fe Institute, }
\centerline{1399 Hyde Park Road, Santa Fe, NM 87501, USA }
\centerline{santiago.elena@uv.es}}

\medskip
{\footnotesize
\centerline{Centre de Recerca Matem\`{a}tica, Campus de Bellaterra, Edifici C,}
\centerline{08193 Bellaterra, Barcelona, Spain}
\centerline{and}
\centerline{Departament de Matem\`{a}tiques, Universitat Aut\`{o}noma de Barcelona,}
\centerline{Campus de Bellaterra, Edifici C, 08193 Barcelona, Spain}
\centerline{akorobeinikov@crm.cat}}

\bigskip

{\centering \textbf{Proposed running head:} Stability of an RNA virus replication model
\vspace{1cm}\par}

{\centering \textbf{AMS Classification (MSC2010)}\\
 92D30 (primary), 34D20, 60H10 (secondary)\par}

{\centering \vspace{1cm}\par}


%


\begin{abstract}

Compared to the replication of double-stranded RNA and DNA viruses, the replication of single-stranded viruses requires the production of a number of intermediate strands that  serve as templates
for the synthesis of genomic-sense strands. Two theoretical extreme mechanisms for replication for such single-stranded viruses have been proposed;  one extreme being represented by the so-called linear stamping machine
and the opposite extreme by the exponential growth. Of course, real systems are more complex and examples have been described in which a combination of such extreme mechanisms can also occur:
a fraction of the produced progeny  resulting from a stamping-machine type of replication that uses
the parental genome as template, whereas others fraction of the progeny  results from the replication of other progeny genomes.
Mart\'{\i}nez et al.~\cite{Matrinez_et-al-2011},  Sardany\'{e}s at al.~\cite{Sardanyes-et-al-2012} and Forn\'{e}s et al.~\cite{Fornes-et-al-2019} suggested and analyzed a deterministic model of single-stranded RNA (ssRNA) virus intracellular replication  that incorporated variability
in the replication mechanisms.

To explore how stochasticity can affect this mixed-model principal properties, in this paper we consider
the stability of a stochastically perturbed model of ssRNA virus replication within a cell.
Using the direct Lyapunov method, we found sufficient conditions for the stability in probability
of equilibrium states for this model. This result confirms that this heterogeneous model of
single-stranded RNA virus replication is stable with respect to stochastic perturbations of
the environment.

\medskip

\end{abstract}

\vspace{10mm}
\textbf{Acknowledgments:}

Santiago F. Elena is supported by Spain's Ministerio de Ciencia, Innovaci\'{o}n y Universidades
grant BFU2015-65037-P.

Andrei Korobeinikov is supported by the Spain's Ministerio de Ciencia, Innovaci\'{o}n y Universidades
grant MTM2015-71509-C2-1-R.



\section{Introduction}

RNA viruses are the most abundant pathogens of bacteria, plants, animals and humans and the
largest source of new emerging infectious diseases.  Moreover, the genomic simplicity,
combined with the high levels of mutability and evolvability typical for RNA viruses,
makes them excellent experimental models in evolutionary biology and provides further motivations
for their study.

The fast replication, large population sizes and high mutation rates that are typical
for RNA viruses lead to a high diversity of genotypes in the viral population replicating within
a single host. This is usually referred to in the literature as viral
quasispecies~\cite{Domingo-1976,Domingo-1978,Eigen,Eigen-Schuster} and can be roughly defined
as a master sequence surrounded by a cloud of mutant genomes that is maintained in a balance
between mutation and selection.
The generation and maintenance of master and mutant genomes distribution in a quasispecies
depends, to a large extent, on a replication pattern that each particular RNA virus adopts
~\cite{Sardanyes2010,Elena-2010}.  The majority of theoretical quasispecies models in the literature
assumed an exponential or, more generally, geometric replication mechanism (hereafter GR)
of RNA populations.  For single-stranded RNA (ssRNA) viruses, the replication of the viral
genome is the result of an RNA-to-RNA transcription process.  In the case of ssRNA viruses,
this process includes the synthesis of an intermediate antigenomic RNA of complementary
polarity that serves as the template for transcription of new genomic RNA progeny.
For the geometric replication, both the genomic and antigenomic viral strands are used
with equal probability as templates for replication.  That is, in the GR mode mutant genomes
also serve as templates for replication, and, hence, transcription errors are geometrically
amplified. As a result, the mutational load accumulated in the replicating population
is very large. The distribution of the number of mutants per infected cell for the GR is
described by the Luria-Delbr\"{u}ck distribution~\cite{Dewanji-2005} and characterized by a variance larger than the mean. This distribution has
been observed in experimental studies done with bacteriophage \emph{T}2~\cite{Luria-1951(50)}.

An alternative replication mode exhibited by ssRNA viruses is the so-called stamping machine
replication mode (SMR). For this replication model, the initially infecting genomic strand
is used as the template for the production of one or few antigenomic strands, which are
subsequently used as templates for the generation of all the progeny of the genomic-sense strands.
This progeny is to be encapsidated and released by the cell to continue the infection process.
For the SMR mode, the mutation frequency remains approximately constant and proportional to the error
rate of the viral replicase, and the number of mutant genomes per infected cell follows a simple
Poisson distribution. Such a distribution of mutant genomes was observed for bacteriophages
 $\phi X$174~\cite{Denhardt-1966(52)} and $Q \beta$~\cite{GarciaVillada-2012,Bradwell-2013}.

Combinations of these two modes and intermediate modes of replication, where a fraction of
genomic-sense progeny may be also used as template for replication, have been also observed.
For example, a recent study showed that \emph{Turnip mosaic virus} (TuMV) exhibited such an
intermediate mode of replication, with most replication events following a SMR mode~\cite{Matrinez_et-al-2011}.
Likewise, a distribution of mutants that slightly differed from the Poisson distribution was
also observed for bacteriophage $\phi$6 \cite{Chao2002}.

Recently, a mathematical model of within-cell ssRNA virus replication that accounts for both,
GR and SMR modes, as well as any intermediate replication modes, was suggested~\cite{Matrinez_et-al-2011,Sardanyes-et-al-2012}.
This model provides a reasonably good fit to the observed TuMV data and is equally applicable
to the positive- and negative-sense ssRNA viruses. The model was shown to be globally asymptotically
stable~\cite{Sardanyes-2014} and structurally stable. While this model brought important insights
into the problem, its weakness is that it is related to a comparatively simple deterministic
model with a rather limited ability to reflect the complexity of real-life biological systems.
In such a situation, the use of stochastic models, where the complexity and uncertainty of
real-life systems can be to some extend captured by stochastic perturbations, appears to be
a natural choice.

Stochasticity is an unavoidable factor at the early stages of virus infection.  Since infection starts with one or very few viral particles entering the cell, the likelihood that
each viral component finds its right cellular partners in a molecularly crowed within-cell environment is relatively low.  Consequently, the outcome of infection is highly likely to be affected by the variability in the initial molecular interactions between the virus and host cell resources.  The impact of this inherent stochasticity in ssRNA virus replication has not been properly addressed yet from a theoretical perspective, with the recent paper by Sardany\'{e}s et al.~\cite{Sardanyes-et-al-2018} being a first attempt in this direction.  This study confirmed the important role of noise in the amplification of ssRNAs and found that it may even induced bistability in the system, a theoretical prediction that has not yet been observed in real viral systems.  The objective of this contribution is to explore the model stability with respect to stochastic
perturbations. To address this issue, in this paper we formulate a stochastically perturbed
version of the model and establish its stability applying the direct Lyapunov method.

\section{Deterministic model and its basic properties}

The within-cell amplification dynamics of genomic and antigenomic strands can be described
by the following differential equations~\cite{Matrinez_et-al-2011,Sardanyes-et-al-2012,Fornes-et-al-2019}:
\begin{equation}\label{model}
\aligned
\dot p=&rm\left(1-\frac{p+m}{K}\right)-\delta p,\\
\dot m=&\alpha rp\left(1-\frac{p+m}{K}\right)-\sigma m.
\endaligned
\end{equation}
In this model, $p(t)$  and $m(t)$ are concentrations of genomic ($p$) and antigenomic ($m$)
viral RNAs, respectively. Positive parameters $r$ and $\alpha r$ are the amplification rates
of the viral genomic and antigenomic RNA molecules by the viral replicase. Factor $\alpha \in (0,1]$
allows to represent the replication rate of antigenomic strands as a fraction of the replication
rate of the genomic strands. When $\alpha=1$, both strands replicate at the same rate and thus
the mode of replication is purely GR. As $\alpha \rightarrow 0$, amplification is closer to
the SMR mode. $K>0$ is the cellular carrying capacity that limits the total amount of RNA
molecules that can be produced in a cell, and $\delta$ and $\sigma$  are the genomic degradation
rates of sense and antisense RNA molecules, respectively. The system is defined in triangle
$$
\Omega=\{(p,m) \in R^2: 0 \leq p+m \leq K\},
$$
which is system phase space.

The original model was formulated for a positive-sense ssRNA virus, but it is symmetric, and
hence it is equally applicable to negative-sense ssRNA viruses, and this is why here we use
the more general terminology genomic/antigenomic or sense/antisense.

\begin{lemma}
\label{lem-invariance}
Triangle $ \Omega=\{(p,m) \in R^2: 0 \leq p+m \leq K\} $ is a positive invariant set of system (\ref{model}).
\end{lemma}

\begin{proof}
Lyapunov function $M=p+m$ satisfies
\begin{eqnarray*}
\frac{dM}{dt} & = & rm\left(1-\frac{p+m}{K}\right)-\delta p+\alpha rp\left(1-\frac{p+m}{K}\right)-\sigma m.
\end{eqnarray*}
It is easy to see that $\frac{dM}{dt}<0$ holds for all $p+m\geq K$, and hence segment $p+m=K$
of the boundary $\overline{\Omega}$ is impenetrable from $\Omega$.

Furthermore, at $p=0$,
\[
\dot{p}=rm\left(1-\frac{m}{K}\right)\geq0
\]
holds for all $m\leq K$, and hence segment $p=0$ of the boundary is impenetrable from $\Omega$
as well. Likewise,
\[
\dot{m}=\alpha rp\left(1-\frac{p}{K}\right)\geq0
\]
holds at $m=0$ for all $p\leq K$, and hence segment $m=0$ of the boundary is impenetrable
from $\Omega$. This completes the proof.
\end{proof}

The positive invariance means that any trajectory initiated in triangle $\Omega$ remains there
indefinitely.

\subsection{Equilibrium states}

It is easy to see that the origin $E_0=(0,0)$ is an equilibrium state of the model. Apart
from the origin, the model can have a positive equilibrium state $E_+=(p^*,m^*)$, where both
sub-populations coexist.

Indeed, equilibrium states of this system satisfy equalities
\begin{equation}\label{equilibria}
\aligned
rm(1-b(p+m))=&\;\delta p, \\
\alpha rp(1-b(p+m))=&\;\sigma m,
\endaligned
\end{equation}
(where $b=1/K$), and, hence, $\sigma m^2=\alpha\delta p^2$ holds at an equilibrium state.
Let us suppose that equilibrium levels $p^*$ and $m^*$ are nonzero and are of the same signs
(are positive, as we are interested in positive $p$ and $m$ only). Then
\begin{equation}\label{m=AAAp}
m=\sqrt{\frac{\alpha\delta}{\sigma}}p.
\end{equation}
Denoting
\begin{equation}\label{R0}
R_0=r\sqrt{\frac{\alpha}{\delta\sigma}}
\end{equation}
and substituting (\ref{m=AAAp}) into (\ref{equilibria}), we obtain equality
$$
R_0\left(1-bp\left(1+\frac{\delta}{r}R_0\right)\right)=1.
$$
This immediately yields equilibrium state $E_+=(p^*,m^*)$, where
\begin{equation}\label{p*,m*}
p^*={1-R_0^{-1}\over b\Big(1+{\delta\over r}R_0\Big)},\qquad
m^*={{\delta\over r}(R_0-1) \over b\Big(1+{\delta\over r}R_0\Big)}.
\end{equation}
Please note that
$$
p^* + m^* = \frac{R_0^{-1}+{\delta\over r}}{b\Big(1+{\delta\over r}R_0\Big)}(R_0-1)
          = K \frac{R_0-1}{R_0} < K .
$$

The value $R_0$ has an obvious biological interpretation: it is the basic RNA reproduction
number, that is an average number of viral RNA produced by a single RNA for its entire lifespan
under the most favorable conditions, that is when $p+m\ll K$ holds. It is obvious, therefore, that
$R_0>1$ is essential for persistence of virus within a cell: it is easy to see in (\ref{p*,m*})
that equilibrium state $E_+$ is located in the positive quadrant of the real plane if and only
if $R_0>1$ and that for all $R_0<1$ values $p^*$ and $m^*$ are negative.
At $R_0=1$ equilibria $E_0$ and $E_+$ merge at the origin; that is, a saddle-node bifurcation
occurs at $R_0=1$.

When $R_0<1$, the origin $E_0=(0,0)$ is a stable node. For all $R_0>1$ point $E_0$ is
a saddle point, whereas $E_+$ is a stable node. Moreover, it can be proven that
for $R_{0}\leq 1$ equilibrium state $E_{0}$ is globally asymptotically stable.

\begin{theorem}
If $R_{0}\leq1$, then equilibrium state $E_{0}$ is globally asymptotically stable in $\Omega$.
\end{theorem}

\begin{proof}
Let us consider Lyapunov function
\[
W=\sigma p+ rm.
\]
The Lyapunov function satisfies
\begin{eqnarray*}
\frac{dW}{dt} & = & \sigma rm\left(1-\frac{p+m}{K}\right)-\sigma \delta p
  + \alpha p r^2 \left(1-\frac{p+m}{K}\right)-r\sigma m\\
 & = & -(r\sigma-r\sigma)m-(\sigma \delta-\alpha r^2)p-(\sigma r m+\alpha r^2p)\frac{p+m}{K}\\
 & = & -\sigma \delta (1-R_{0}^{2})p -r\sigma \left( m+ \frac{\delta}{r} R_{0}^{2} p\right)\frac{p+m}{K}.
\end{eqnarray*}
That is, $R_{0}\leq1$ is sufficient to ensure that $\frac{dW}{dt}<0$ for all $(p,m)\in\Omega$
apart from $E_{0}=(0,0)$ (where $\frac{dW}{dt}=0$).  By Lyapunov asymptotic stability theorem,
and recalling that, by Lemma~\ref{lem-invariance}, $\Omega $ in a positive invariant set of
system (\ref{model}), equilibrium state $E_{0}$ is globally asymptotically stable (in $\Omega$).
\end{proof}

The global asymptotic stability of equilibrium state $E_{+}$ for $R_{0}>1$ was proven by J.
Sardany\'{e}s et al.~\cite{Sardanyes-2014} using Dulac's criterion. Indeed, let us consider
the divergence $\nabla f$ of vector field $f=(\dot{p},\dot{m})$ defined by equations~(\ref{model}).
It is easy to see that
$$
\nabla f= \frac{\partial \dot{p}}{\partial p}+\frac{\partial \dot{m}}{\partial m}
= - \frac{r}{K}(m + \alpha p)-\delta -\sigma.
$$
That is, $\nabla f<0$ in $\Omega$ (in fact, the inequality holds for all $p,m\geq 0$), and hence,
by Dulac's criterion, there is no limit cycles in $\Omega$.  Taking into consideration the
positive invariance of set $\Omega$, that points $E_{0}$ and $E_{+}$ are the only equilibrium
states of system (\ref{model}) in $\Omega$, and that for all $R_{0}>1$ $E_{0}$ is unstable
whereas $E_{+}$ is locally stable, by the Andronov theorem~\cite{Andronov}, the
equilibrium state $E_{+}$, when it exists in $\Omega$ (that is for all $R_{0}>1$), is globally
asymptotically stable.

Please note that the global asymptotic stability in $\Omega$ of equilibrium state $E_{0}$
for $R_{0}\leq 1$ can be proven by the same arguments.

\subsection{Equilibria with negative coordinates}

The case when equilibrium levels $p$ and $m$ are nonzero of different signs is not of practical
relevance for the model. However, these can help to better understand model global dynamics.
If $p$ and $m$ are of different signs, then
$$
m=-\sqrt{\alpha\delta\over\sigma}p.
$$
Substituting this equality into (\ref{equilibria}) we have
$$
-R_0\left(1-bp\left(1-\delta R_0 /r \right)\right)=1.
$$
This yields another equilibrium state $E_{(+-)}=(p_{(+-)},$ $m_{(+-)})$ with
$$
p_{(+-)}={1+R_0^{-1} \over b\Big(1-{\delta\over r}R_0\Big)},\quad
m_{(+-)}=-{\delta\over r}{R_0+1 \over b\Big(1-{\delta\over r}R_0\Big)}.
$$
It is easy to see that $ p_{(+-)}>0$ and $ m_{(+-)}<0 $ when  $ R_0 < r/ \delta$, or $ p_{(+-)}<0$
and $m_{(+-)}>0 $ when  $R_0 > r/ \delta $. At $ R_0 = r/ \delta $ both, $p_{(+-)}$ and $m_{(+-)}$,
take infinite values.

\section{Stochastic perturbations, centralization and linearization}

Let us assume that the considered system (\ref{model}) is exposed to stochastic perturbations
that are of the type of white noise and are proportional to a deviation of the system state
$(p,m)$ from the equilibrium $E_+=(p^*,m^*)$. In this case we obtain the system of It\^o's stochastic
differential equations
\begin{equation}
\aligned
\dot p=&\;rm\left(1-b(p+m)\right)-\delta p+\omega_1(p-p^*)\dot w_1,\\
\dot m=&\;\alpha rp\left(1-b(p+m)\right)-\sigma m+\omega_2(m-m^*)\dot w_2,
\endaligned
\label{stochastic-model}
\end{equation}
where $\omega_1$, $\omega_2$ are constants and $w_1(t)$, $w_2(t)$ are mutually independent
standard Wiener processes \cite{Gikhman,Shaikhet2}. Please note that for the proposed type
of stochastic perturbations the equilibrium state $E_+=(p^*,m^*)$
of the deterministic model (\ref{model}) is also a solution to the system of stochastic
differential equations (\ref{stochastic-model}).  This type of stochastic
perturbation was firstly introduced in \cite{Beretta} and thereafter was extensively applied
to a variety of mathematical models (see \cite{Shaikhet2,Shaikhet1,Shaikhet-Kor-2016}
and references therein).

To centralize system (\ref{stochastic-model}) around equilibrium state $E_+=(p^*,m^*)$,
we substitute $p=x_1+p^*$ and $m=x_2+m^*$ into (\ref{stochastic-model}).  Then, using (\ref{equilibria}),
we obtain the following system of nonlinear stochastic differential equations:
\begin{equation}\label{nonlin}
\aligned
\dot x_1(t)=&a_{11}x_1(t)+a_{12}x_2(t)\\
&-br(x_1+x_2)x_2+\omega_1x_1(t)\dot w_1(t),\\
\dot x_2(t)=&a_{21}x_1(t)+a_{22}x_2(t)\\
&-\alpha br(x_1+x_2)x_1+\omega_2x_2(t)\dot w_2(t),
\endaligned
\end{equation}
where
\begin{equation}\label{matrix_A}
\gathered
a_{11}=-[brm^*+\delta],\quad a_{12}=r[1-b(p^*+2m^*)],\\
a_{21}=\alpha r[1-b(2p^*+m^*)], \quad
a_{22}=-\left[\alpha brp^*+\sigma\right].
\endgathered
\end{equation}
The equilibrium state $E_+=(p^*,m^*)$ of system \eqref{stochastic-model} is stable if and
only if the zero solution of system (\ref{nonlin}) is stable.
(Please note that, for the sake of simplicity, here and below stochastic differential
equations are written in the form of derivatives, i.e. $\dot x = a+b \dot w$,
understanding by this the differential form $dx=adt+bdw$.)

\section{Stability of the stochastically perturbed model}

\subsection{Definitions and auxiliary statements}

\begin{definition} The zero solution of  system (\ref{linear-model}) is called mean square
stable, if for each $\varepsilon>0$ there exists $\delta>0$ such that
$\bold E|y(t,y_0)|^2<\varepsilon$ $(y=(y_1,y_2))$
holds for all $t\ge0$, provided that $\bold E|y_0|^2<\delta$. The solution is asymptotically
mean square stable, if it is mean square stable, and, for any initial value $y_0$,
$\lim_{t\to\infty}\bold E|y(t,y_0)|^2=0$.
\end{definition}

\begin{definition} The zero solution of  system (\ref{nonlin}) is called stable in probability,
if, for any $\varepsilon_1>0$ and $\varepsilon_2>0$, there exists $\delta>0$ such that, for
any initial value $x_0$, solution $x(t,x_0)$ $(x=(x_1,x_2))$ to equation (\ref{nonlin}) satisfies
condition $\bold P\{\sup_{t\ge0}|x(t,x_0)|>\varepsilon_1\}<\varepsilon_2$, where $\bold P\{|x_0|<\delta\}=1$.
\end{definition}


\begin{definition}
With the It\^o stochastic differential equation (see \cite{Gikhman})
\begin{equation}
\gathered
dx(t)=a_1(t,x(t))dt+a_2(t,x(t))dw(t),\\
t\ge0, \qquad  x(t)\in\bold R^n, \qquad x(0)=x_0.
\endgathered
\label{A.1}
\end{equation}
is associated the generator
\begin{equation}
LV(t,x)=V_t(t,x)+\nabla V'(t,x)a_1(t,x)+\frac{1}{2}Tr[a'_2(t,x)\nabla^2V(t,x)a_2(t,x)],
\label{A.2}
\end{equation}
where
$$
\gathered
V_t={\partial u(t,x)\over\partial t},\qquad
\nabla V=\left({\partial V(t,x)\over\partial x_1},...,{\partial V(t,x)\over\partial x_n}\right),\\
\nabla^2V=\left({\partial^2 V(t,x)\over\partial x_i\partial x_j}\right),\quad i,j=1,...,n.
\endgathered
$$
\end{definition}

\begin{remark}\label{R4.1} The order of nonlinearity of the nonlinear system of
stochastic differential equations (\ref{nonlin}) is higher than one. For such a system,
sufficient conditions for the asymptotic mean square stability of the zero solution of
the linear part of system (\ref{nonlin}), that is, in this case, of linear system
\begin{equation}\label{linear-model}
\gathered
\dot y_1(t)=a_{11}y_1(t)+a_{12}y_2(t)+\omega_1y_1(t)\dot w_1(t),\\
\dot y_2(t)=a_{21}y_1(t)+a_{22}y_2(t)+\omega_2y_2(t)\dot w_2(t),
\endgathered
\end{equation}
are, at the same time, sufficient conditions for the stability in probability of the zero
solution of nonlinear system (\ref{nonlin}) (see~\cite{Shaikhet2}, p.~130).
\end{remark}


Let us denote
\begin{equation}\label{td}
\gathered
A=(a_{ij}),\quad \gamma_i=\frac{1}{2}\omega_i^2,\quad i,j=1,2,\\
Tr(A)=a_{11}+a_{22},\quad
\det(A)=a_{11}a_{22}-a_{12}a_{21},\\
A_1=\det(A)+a_{11}^2,\quad A_2=\det(A)+a_{22}^2,\\
y(t)=\begin{bmatrix} y_1(t) \\ y_2(t) \end{bmatrix}, \quad
S_1=\begin{bmatrix} \omega_1 & 0 \\ 0 & 0 \end{bmatrix}, \quad
 S_2=\begin{bmatrix} 0 & 0 \\ 0 & \omega_2 \end{bmatrix},
\quad Q=\begin{bmatrix} q & 0 \\ 0 & 1 \end{bmatrix},\quad q>0,
\endgathered
\end{equation}
and represent the system (\ref{linear-model}) in the matrix form
\begin{equation}\label{limat}
\dot y(t)=Ay(t)+\sum^2_{i=1}S_iy(t)\dot w_i(t).
\end{equation}

\begin{lemma}
\label{L4.1}
Suppose that $Tr(A)<0$, $\det(A)>0$ and
\begin{equation}\label{trd}
\gamma_1<\frac{|Tr(A)|\det(A)}{A_2},\qquad
\gamma_2<\frac{|Tr(A)|\det(A)-A_2\gamma_{1}}{A_1-|Tr(A)|\gamma_{1}}.
\end{equation}
Then the zero solution of system (\ref{linear-model}) is asymptotically mean square stable.
\end{lemma}

\begin{proof}
(See \cite{Shaikhet2})
Conditions $Tr(A)<0, \det(A)>0$ ensure that matrix equation $PA+A'P=-Q$ has
a positive definite solution $P=\|p_{ij}\|$.
For matrix $A$ the elements of matrix $P$ are
\begin{equation}\label{pii}
p_{11}=\frac{A_2q+a^2_{21}}{2|Tr(A)|\det(A)},\quad
p_{22}=\frac{A_1+a^2_{12}q}{2|Tr(A)|\det(A)},\quad
p_{12}=\frac{a_{12}a_{22}q+a_{21}a_{11}}{2|Tr(A)|\det(A)}.
\end{equation}
Let $L$ denote the generator of equation \eqref{limat}.
Lyapunov function $ v(y)=y'Py $ satisfies
\begin{equation}\label{Lv}
\aligned
Lv(y(t))=&y'(t)\left(PA+A'P+\sum^2_{i=1}S'_iPS_i\right)y(t)\\
=&(-q+2p_{11}\gamma_1)y^2_1(t)+(-1+2p_{22}\gamma_2)y^2_2(t).
\endaligned
\end{equation}
That is, if here exists $q>0$ such that
\begin{equation}\label{ineq}
-q+2p_{11}\gamma_1<0, \quad -1+2p_{22}\gamma_2<0
\end{equation}
hold, then $Lv(y(t))$ is negative definite, whereas $v(y)$ is positive definite, and,
therefore, the zero solution of equation \eqref{limat} is asymptotically mean square stable.

Substituting \eqref{pii} into \eqref{ineq}, we obtain
\begin{equation*}
\frac{(A_2q+a^2_{21})\gamma_1}{|Tr(A)|\det(A)}<q, \quad  \frac{(A_1+a^2_{12}q)\gamma_2}{|Tr(A)|\det(A)}<1.
\end{equation*}
From these inequalities we have
\begin{equation}\label{q}
\frac{a^2_{21}\gamma_1}{|Tr(A)|\det(A)-A_2\gamma_1}<q<\frac{|Tr(A)|\det(A)-A_1\gamma_2}{a^2_{12}\gamma_2}.
\end{equation}
That is, if
\begin{equation}\label{q0}
\frac{a^2_{21}\gamma_1}{|Tr(A)|\det(A)-A_2\gamma_1}<\frac{|Tr(A)|\det(A)-A_1\gamma_2}{a^2_{12}\gamma_2}
\end{equation}
holds, then there exists $q>0$ such that \eqref{q}, and, therefore, \eqref{ineq} holds.

Inequality \eqref{q0} holds by Lemma's hypotheses.
Indeed, by the first of the inequalities \eqref{trd}, we can rewrite inequality \eqref{q0} as
\begin{equation}\label{q1}
a^2_{12}a^2_{21}\gamma_1\gamma_2<(|Tr(A)|\det(A))^2-|Tr(A)|\det(A)(A_1\gamma_2+A_2\gamma_1)+A_1A_2\gamma_1\gamma_2.
\end{equation}
Immediately, from the definitions of $A_1$ and $A_2$, we have
\begin{equation}\label{A12}
A_1A_2=|Tr(A)|^2\det(A)+a^2_{12}a^2_{21}. 
\end{equation}
Hence, it suffices to show that
\begin{equation}\label{g2q}
0<|Tr(A)|\det(A)-A_2\gamma_1-(A_1-|Tr(A)|\gamma_1)\gamma_2
\end{equation}
holds. From \eqref{A12},
$$
A_1A_2=|Tr(A)|^2\det(A)+a^2_{12}a^2_{21}\ge|Tr(A)|^2\det(A),
$$
and hence 
\begin{equation}\label{g2}
\gamma_1<\frac{|Tr(A)|\det(A)}{A_2}\le\frac{A_1}{|Tr(A)|}.
\end{equation}
Therefore, inequality \eqref{g2q} is equivalent to the second equation in \eqref{trd}.
The proof is completed.
\end{proof}

%

\begin{remark} Note that \eqref{g2} ensures that
the right-hand part of the second inequality in (\ref{trd}) is positive.
\end{remark}

\subsection{Stability of equilibrium $E_+$}

\begin{theorem} Suppose that $R_0>1$ and that conditions (\ref{trd}) hold.  Then equilibrium
state $E_{+}$ of the stochastically perturbed system (\ref{stochastic-model}) is stable in
probability.
\end{theorem}

\begin{proof}
By Remark \ref{R4.1} and Lemma \ref{L4.1}, it is sufficient to show that $Tr(A)<0$ and $\det(A)>0$
hold for equilibrium state $E_{+}$. For system (\ref{linear-model}),
$$
Tr(A)=-(br(m^*+\alpha p^*)+\delta +\sigma ),
$$
and hence $Tr(A)<0$ for all positive (and hence all practically relevant) values of system (\ref{linear-model})
parameters.  Furthermore,
\begin{equation}\label{da1}
\aligned
\det(A)&=\;\delta\sigma+br(\alpha\delta p^*+\sigma m^*)\\
&-\alpha r^2[1+2b^2(p^*+m^*)^2-3b(p^*+m^*)].
\endaligned
\end{equation}
Here, by (\ref{p*,m*}),
\begin{equation}\label{da2}
\aligned
b(\alpha\delta p^*+\sigma m^*)= &\frac{\alpha\delta\left(1-R_0^{-1}\right)+\frac{\sigma\delta}{r}(R_0-1)}
{1+\frac{\delta}{r}R_0}\\
=&\frac{\delta(R_0-1)(\alpha R_0^{-1}+\frac{\sigma}{r})}{1+\frac{\delta}{r}R_0}
\endaligned
\end{equation}
and
\begin{equation}\label{da3}
\aligned
b(p^*+m^*)=&\frac{1-R_0^{-1}+\frac{\delta}{r}R_0(1-R_0^{-1})}{1+\frac{\delta}{r}R_0}\\
=&1-R_0^{-1}.
\endaligned
\end{equation}
Substituting (\ref{da2}) and (\ref{da3}) into (\ref{da1}) and using the definition of $R_0$ (\ref{R0}), we obtain
$$
\aligned
\det(A)=&\delta\sigma+{r\delta( R_0-1)\Big({\alpha \over R_0}+{\sigma\over r}\Big)\over 1+{\delta\over r}R_{0}}\\
&-\alpha r^2\left[1+2\left(1-{1\over R_0}\right)^2-3\left(1-{1 \over R_0}\right)\right]\\
=&{\delta(\alpha r+\sigma R_0)\over1+{\delta\over r}R_{0}}-\alpha r^2\left({2\over R_0^2}-{1\over R_0}\right)\\
=& \delta \sigma {\frac{\alpha r}{\sigma}+R_0\over1+{\delta\over r}R_{0}}-\frac{\alpha r^2}{R_0^2}\left( 2  - R_0 \right)\\
=& \delta \sigma {{\delta\over r}R_{0}^2+ R_0\over1+{\delta\over r}R_{0}}-\delta \sigma \left( 2  - R_0 \right)\\
=&2\delta\sigma(R_{0}-1).
\endaligned
$$
Hence, $R_0>1$ is necessary and sufficient condition for the positivity of $\det(A)$. The proof
is now completed.
\end{proof}

\textbf{Example}: Let
$\alpha=0.0743$, $r=0.1211$, $\delta=0.0049$, $K=4.694\cdot 10^7$ and $\sigma=0.0121$.
(These values corresponds to TuMV; see Table~1 in~\cite{Matrinez_et-al-2011}.)
For these values, $R_0=4.287$ and $E_+=(30670385, 5320090)$.
Furthermore,
$$
\gathered
a_{11}=-0.01862524,\quad a_{12}=0.01452332,\\
a_{21}=-0.00378021,\quad a_{22}=-0.01797908,\\
Tr(A)=-0.03660432,\quad \det(A)=0.00038977,
\endgathered
$$
and, hence,
$$
\gamma_1<0.02000961,\quad
\gamma_2<0.01947893\frac{0.02000961-\gamma_1}{0.02012510-\gamma_1}.
$$

\subsection{Stability of equilibrium state $E_0$}

By (\ref{td}) and (\ref{matrix_A}), for equilibrium state $E_0$ we have
$$
Tr(A)=-(\delta+\sigma),\qquad \det(A)=\delta\sigma(1-R_{0}^2).
$$
It is easy to see that $Tr(A)<0$ holds for all positive values of the system parameters,
whereas $R_0<1$ ensures that $\det(A)>0$ holds as well. Hence, we just proved the following
theorem:

\begin{theorem} If $R_0<1$ and
$$
\aligned
\gamma_1<&{\delta(\delta+\sigma)(1-R_{0}^2)\over\sigma+\delta(1-R_{0}^2)},\\
\gamma_2<&{\sigma[\delta(\delta+\sigma)(1-R_{0}^2)-(\sigma+\delta(1-R_{0}^2))\gamma_1]\over \delta(\delta+\sigma(1-R_{0}^2))-(\delta+\sigma)\gamma_1},
\endaligned
$$
then equilibrium state $E_0$ of the stochastically perturbed system (\ref{stochastic-model})
is stable in probability.
\end{theorem}

%
%
%

\section{Conclusion}

To study impact of stochastic perturbations on the dynamics of a deterministic RNA virus
amplification model (\ref{model}) that was earlier proposed
in~\cite{Matrinez_et-al-2011,Sardanyes-et-al-2012}, in this paper we  considered stability of
a stochastic model, which is a straightforward extension of the original
model (\ref{model}).  To analyze the properties of this stochastically perturbed model, 
we used the direct Lyapunov methods. The results of our analysis show that the positive equilibrium
state of the stochastically perturbed model is stable in probability for all $R_0>1$, whereas
the virus-free equilibrium state $E_{0}$ at the origin is stable in probability when $R_{0}<1$.
These results confirm that deterministic model (\ref{model}) is fairly robust with respect to
stochastic perturbations, and that it is unlikely that perturbations of practically realistic
magnitudes would be able to significantly change the model dynamics.

\end{document}